\documentclass[a4paper]{amsart}

\usepackage{graphicx, psfrag, layout, appendix, subfigure}
\usepackage{fancyhdr}

\usepackage{amssymb, amsthm, amsmath, amscd, amsfonts, mathabx, mathrsfs, latexsym}

\usepackage{enumitem}
\usepackage{color}
\usepackage{url}

\usepackage{setspace}
\usepackage{changepage}

\usepackage{leftidx}

\usepackage{bbold}

\usepackage[active]{srcltx}

\usepackage{calc}

\usepackage{pb-diagram,pb-xy}

\usepackage{times}

\usepackage{verbatim}
\usepackage[all,cmtip,arrow,matrix,curve,tips,frame]{xy}
\usepackage[
	setpagesize=false,
	colorlinks=true,
	linkcolor=blue,
	pdfencoding=auto,
	psdextra,
]{hyperref}
\usepackage{cleveref}
\xyoption{matrix}

\usepackage{xcolor}
\hypersetup{
    colorlinks,
    linkcolor={red!50!black},
    citecolor={blue!50!black},
    urlcolor={blue!80!black}
}

\usepackage{ifthen}
\usepackage{color}
\usepackage{amsxtra}
\usepackage{amstext}
\usepackage{stmaryrd}
\usepackage{mathrsfs}
\usepackage{epsfig}
\usepackage{pifont}
\usepackage{charter}
\usepackage{avant}
\usepackage{courier}
\usepackage[T1]{fontenc}
\usepackage{textcomp}
\usepackage{bbm}
\usepackage{tikz-cd}
\usepackage{pdfcomment}
\pdfcommentsetup{color=yellow, icon=Note, disable=false, date={.}}

\setlist[enumerate,1]{label=(\arabic*), ref=(\arabic*)}
\setlist[enumerate,3]{label=(\roman*), ref=(\roman*)}

\numberwithin{equation}{section}

\theoremstyle{plain}
\newtheorem{thm}{Theorem}[section]
\newtheorem*{thm*}{Theorem}

\newtheorem{lem}[thm]{Lemma}

\newtheorem*{lem*}{Lemma}

\newtheorem*{prop*}{Proposition}

\newtheorem{cor}[thm]{Corollary}

\newtheorem*{cor*}{Corollary}

\newtheorem*{claim*}{Claim}

\newtheorem{conj}[thm]{Conjecture}
\newtheorem*{conj*}{Conjecture}

\theoremstyle{definition}

\newtheorem*{defn*}{Definition}

\newtheorem*{ques*}{Question}

\newtheorem*{exa*}{Example}

\newtheorem*{acknowledgements}{Acknowledgements}

\theoremstyle{remark}

\newtheorem*{rmk*}{Remark}

\numberwithin{figure}{section}
\numberwithin{table}{section}
\numberwithin{equation}{section}

\def \CC {\mathbb{C}}

\def \NN {\mathbb{N}}

\def \PP {\mathbb{P}}
\def \QQ {\mathbb{Q}}

\def \ZZ {\mathbb{Z}}

\def \Bcal {\mathcal{B}}

\def \Ecal {\mathcal{E}}

\def \Ocal {\mathcal{O}}

\def \hbar {\bar{h}}

\def \hbf {\mathbf{h}}

\title{A non-vanishing conjecture for cotangent bundles on elliptic surfaces}
\author{Haesong Seo}
\address{Department of Mathematical Sciences, KAIST, 291 Daehak-ro, Yuseong-gu, Deajeon, 34141, Republic of Korea
}
\email{hss21@kaist.ac.kr}


\usepackage[
    commandnameprefix=always,
    deletedmarkup=sout,
    commentmarkup=uwave,
    authormarkuptext=name
]{changes}

\begin{document}

\maketitle

\begin{abstract}
    In this paper, we prove the non-vanishing conjecture for cotangent bundles on isotrivial elliptic surfaces.
    Combined with the result by H\"{o}ring and Peternell, it completely solves the question for surfaces with Kodaira dimension at most $1$.
\end{abstract}


\section{Introduction} \label{sec:intro}

The non-vanishing conjecture is one of the main ingredients in the minimal model program, which is stated as follows: if $X$ is a smooth projective variety with $K_X$ pseudoeffective, then $H^0(X,mK_X) \neq 0$ for some $m > 0$.
H\"oring and Peternell \cite{HP2020} suggested its generalization to the cotangent bundles; recall that a vector bundle $\Ecal$ on $X$ is pseudoeffective (resp. big) if the tautological line bundle $\Ocal_{\PP(\Ecal)}(1)$ on the projectivization $\PP(\Ecal)$ is pseudoeffective (resp. big).

\begin{conj} \label{conj:non-vanishing-conjecture-for-cotangent-bundles}
    Let $X$ be a smooth projective variety, and let $1 \leq q \leq \dim X$.
    If $\Omega_X^{q}$ is pseudoeffective, then $H^0(X, S^m\Omega_X^q) \neq 0$ for some $m > 0$.
\end{conj}

The corresponding version for tangent bundles on surfaces has been widely studied.
For example, H\"oring, Liu and Shao \cite{HLS2022} proved that for a smooth del Pezzo surface $S$, the tangent bundle $T_S$ is pseudoeffective (resp. big) if and only if the degree $d = K_S^2$ is at least $4$ (resp. at least $5$).
Also, H\"{o}ring and Peternell \cite{HP2021} provided a splitting structure of the tangent bundle of a smooth non-uniruled projective surface $S$ when $T_S$ is pseudoeffective.
Jia, Lee and Zhong \cite{JLZ2023} showed that if $S$ is a smooth non-uniruled projective surface, then $T_S$ is pseudoeffective if and only if $S$ is minimal and $c_2(S) = 0$.
In particular, the non-vanishing conjecture for tangent bundles holds for del Pezzo surfaces (cf. \cite[Theorem~1.2]{HLS2022}) and for non-uniruled surfaces (cf. \cite[Corollary~1.3]{JLZ2023}).
In addition, the tangent bundle on a surface of general type is not pseudoeffective by \cite[Proposition~3.2]{JLZ2023}.

On the other hand, not much is known for cotangent bundles.
H\"oring and Peternell \cite{HP2020} showed that Conjecture~\ref{conj:non-vanishing-conjecture-for-cotangent-bundles} partially holds for a smooth projective surface $S$ with $\kappa(S) \leq 1$, but the result remains open when $S$ admits an isotrivial elliptic fibration.
Cao and H\"oring \cite{CH2023} proved that if a smooth projective variety $X$ admits an abelian fibration $f:X \rightarrow C$ onto a curve, then Conjecture~\ref{conj:non-vanishing-conjecture-for-cotangent-bundles} holds if $f_{\ast}\Omega_X^1$ has rank $1$, which corresponds to the non-isotrivial one for elliptic surfaces.
For a surface $S$ of general type, the Bogomolov vanishing theorem tells us that $\Omega_S^1$ is big if $c_1^2 - c_2 > 0$, but little is known for the case $c_1^2 - c_2 \leq 0$.

In this paper, we prove Conjecture~\ref{conj:non-vanishing-conjecture-for-cotangent-bundles} for the remaining case of isotrivial elliptic surfaces.
Together with \cite[Proposition~5.4]{HP2020}, it fully answers the question for surfaces $S$ with $\kappa(S) \leq 1$.

\begin{thm}
    Let $f:S \rightarrow B$ be a relatively minimal isotrivial elliptic surface.
    If $\Omega_S^1$ is pseudoeffective, then $H^0(S, S^m \Omega_S^1) \neq 0$ for some $m > 0$.
\end{thm}

Our main theorem can be proven by applying \cite[Proposition~5.2~and~4.6]{HP2020} to the following technical lemma:

\begin{thm}[cf. {\cite[Theorem~6.7]{HP2020}}] \label{thm:main-theorem}
    Let $f:S \rightarrow B$ be a relatively minimal isotrivial elliptic surface.
    Let $D = \sum_{b \in B} [f^{-1}(b) - f^{-1}(b)_{\mathrm{red}}]$, where $f^{-1}(b)_{\mathrm{red}}$ is the reduction of the fibre $f^{-1}(b)$.
    If $\Omega_S^1$ is pseudoeffective, so is $f^{\ast}\Omega_B^1(D)$.
\end{thm}

For a given relatively minimal elliptic surface $f:S \rightarrow B$, let $\lambda^f(B)$ be the Iitaka dimension of the $\QQ$-line bundle $K_B + \sum_{i=1}^s (1-\frac{1}{\nu_i})a_i$, where $f$ has multiple fibres over $a_i \in B$ with multiplicity $\nu_i$.
Notably, a (relatively minimal) isotrivial elliptic surface $f:S \rightarrow B$ with $\lambda^f(B) = 0$ and that is not almost smooth has only one nontrivial symmetric differential (cf. \cite[Section~4, Table III and (H)]{Sakai1979}).

In view of \cite[Lemma~A.1]{HP2020}, one can characterize the pseudoeffectivity of the cotangent bundle on an elliptic surface:

\begin{cor}
    Let $f:S \rightarrow B$ be a smooth elliptic surface.
    Then $\Omega_S^1$ is pseudoeffective if and only if the fundamental group $\pi_1(S)$ is infinite.
\end{cor}

The `if' direction is due to \cite[Lemma~A.1]{HP2020}.
For the converse, suppose that $\Omega_S^1$ is pseudoeffective.
We may assume that $f$ is relatively minimal and not almost smooth.
Theorem~\ref{thm:main-theorem}, as well as \cite[Proposition~5.4]{HP2020}, states that either $g(B) \geq 1$ or $f$ has at least three multiple fibres.
Thus, by \cite[Theorem~IV.9.12]{FK1992}, there is a ramified covering $B' \rightarrow B$ that induces a finite \'etale cover $S' \rightarrow S$ such that the elliptic fibration $S' \rightarrow B'$ has no orbifold divisor.
Now $\pi_1(S') \simeq \pi_1(B')$ is infinite as $g(B') \geq 1$.

The proof of Theorem~\ref{thm:main-theorem} basically uses the same idea as that of \cite[Theorem~6.7]{HP2020}.
However, to deal with singular fibres of types $II$, $III$ and $IV$, we need to work with a birational model $S'$ of $S$ which might not be relatively minimal.
Using an explicit computation, we analyze the local obstruction for symmetric differentials on $S'$ in terms of types of singular fibres.
Then we use the fact from \cite[Proposition~4.1]{HP2020} that the pseudoeffectivity of $\Omega_S^1$ is equivalent to that of $\Omega_{S'}^1$, which completes the proof.

\begin{acknowledgements}
This research is supported by the Institute for Basic Science (IBS-R032-D1).
The author would like to express the gratitude to Prof. Yongnam Lee for his suggestions on research topics and valuable comments, as well as to Prof. Andreas H\"oring for his nice lectures during his visit to IBS-CCG in November 2023.
The author also appreciates Dr. Guolei Zhong for his careful comments and for pointing out missing details in the previous version of this paper.
\end{acknowledgements}


\section{Elliptic surfaces} \label{sec:elliptic-surfaces}

We will work over $\CC$ and follow \cite{Hartshorne1977} for basic notations.
All the varieties are assumed to be reduced and irreducible.

An \textit{elliptic surface} (or an \textit{elliptic fibration}) is a fibration $f:S \rightarrow B$ from a surface to a curve whose general fibre is an elliptic curve.
We typically define elliptic surfaces to be \textit{relatively minimal}, meaning that there are no $(-1)$-curves in the fibres.
There is a classification due to Kodaira \cite{Kodaira1963} for singular fibres: $_mI_b$, $I_b^{\ast}$, $II$, $II^{\ast}$, $III$, $III^{\ast}$, $IV$ and $IV^{\ast}$ for $m \geq 1$ and $b \geq 0$.
Note that the only multiple singular fibres are of type $_mI_b$ for $m \geq 2$.

If a (relatively minimal) elliptic surface is isotrivial, i.e., general fibres are mutually isomorphic, then there are no singular fibres whose monodromy group is infinite.

\begin{lem}[{\cite[Lemma~3.2]{PS2020}}] \label{lem:singular-fibres-in-isotrivial}
    For a relatively minimal isotrivial elliptic fibration, the multiple singular fibres are of type $_{m}I_0$ for $m \geq 2$; the non-multiple singular fibres are of type $I_0^{\ast}$, $II$, $II^{\ast}$, $III$, $III^{\ast}$, $IV$ or $IV^{\ast}$.
\end{lem}

If $f:S \rightarrow B$ is isotrivial, by \cite[Section~2]{Serrano1996}, there exist a smooth curve $C$ and a finite group $G$ such that the following diagram commutes:
 \begin{equation} \label{eqn:isotrivial-elliptic-fibration}
  \begin{tikzcd}
      S' \arrow[rd, "\lambda"] \arrow[d, "\mu"'] & C \times E \arrow[d, "q"] \\
      S \arrow[d, "f"'] & (C \times E)/G \arrow[d] \\
      B \arrow[r, equal] & C/G
  \end{tikzcd}
 \end{equation}
where $E$ is the generic fibre, $G$ acts on $C \times E$ diagonally and $\lambda$ is the minimal resolution of $(C \times E)/G$.
Since $f' = f \circ \mu:S' \rightarrow B$ is not relatively minimal in general, the map $\mu$ might not be an isomorphism.

Note that $G_x$ is cyclic for each $x \in C$ (cf. \cite[p.106,~Corollary]{FK1992}).
Since $E$ is an elliptic curve, $G_x$ acts on $E$ by translation or it fixes a point $e \in E$.
In the latter case, regarding $e$ as an identity element of $E$, one can see that $G_x$ is isomorphic to $\ZZ/2\ZZ$, $\ZZ/3\ZZ$, $\ZZ/4\ZZ$ or $\ZZ/6\ZZ$ by \cite[Corollary~4.7]{Hartshorne1977}.

Choose a small disc $x \in \Delta \subset C$ such that $(C \times E)/G$ is locally isomorphic to $(\Delta \times E)/G_x$.
Then the minimal resolution $S' \rightarrow (C \times E)/G$ is locally isomorphic to that of $(\Delta \times E)/G_x$.
If $G_x$ acts on $E$ by translation, the $f'$-fibre under $x$ is the multiple elliptic $_mI_0$ where $m = |G_x|$, so there is no $(-1)$-curve in the fibre.
On the other hand, if $G_x$ fixes a point $e$ of $E$, then the exceptional divisor of the minimal resolution forms a Hirzebruch-Jung string by \cite[Theorem~III.5.4]{BHPVdV2004}.
Indeed, for each nonzero element $g \in G_x$ and its fixed point $e' \in E$, there is a local coordinate $(s,c)$ at $(x,e')$ such that $g$ acts by $(s,c) \mapsto (\eta_ns, \eta_n^{\pm 1}c)$, where $n$ is the order of $g$ in $G_x$ and $\eta_n$ is a primitive $n$-th root of unity.

\begin{table}[h] \label{table:monodromy-quotient-singularities}
\begin{tabular}{|c|c|c|}
    \hline
    type & $G_x$ & quotient singularities \\
    \hline \hline
    $I_0^{\ast}$ & $\ZZ/2\ZZ$ & $4A_1$\\
    \hline
    $II$ & $\ZZ/6\ZZ$ & $A_{6,1}A_{3,1}A_1$ \\
    \hline
    $II^{\ast}$ & $\ZZ/6\ZZ$ & $A_5A_2A_1$ \\
    \hline
    $III$ & $\ZZ/4\ZZ$ & $2A_{4,1}A_1$ \\
    \hline
    $III^{\ast}$ & $\ZZ/4\ZZ$ & $2A_3A_1$ \\
    \hline
    $IV$ & $\ZZ/3\ZZ$ & $3A_{3,1}$ \\
    \hline
    $IV^{\ast}$ & $\ZZ/3\ZZ$ & $3A_2$ \\
    \hline
\end{tabular}
\caption{Monodromy groups and quotient singularities according to the type of singular fibres.}
\end{table}

The above table lists the monodromy groups and quotient singularities that appear in $(\Delta \times E)/G_x$, based on the types of singular fibres (for more details, see \cite[Table~V.10.5]{BHPVdV2004}).

Finally, when $G_x$ fixes a point of $E$, the corresponding $f'$-fibre has a $(-1)$-curve if and only if the $f$-fibre under $x$ is of type $II$, $III$ or $IV$.
In that case, the map $\mu$ is the composition of the following blowing down procedures.

\begin{figure}[ht]
    \includegraphics[height=2.4cm]{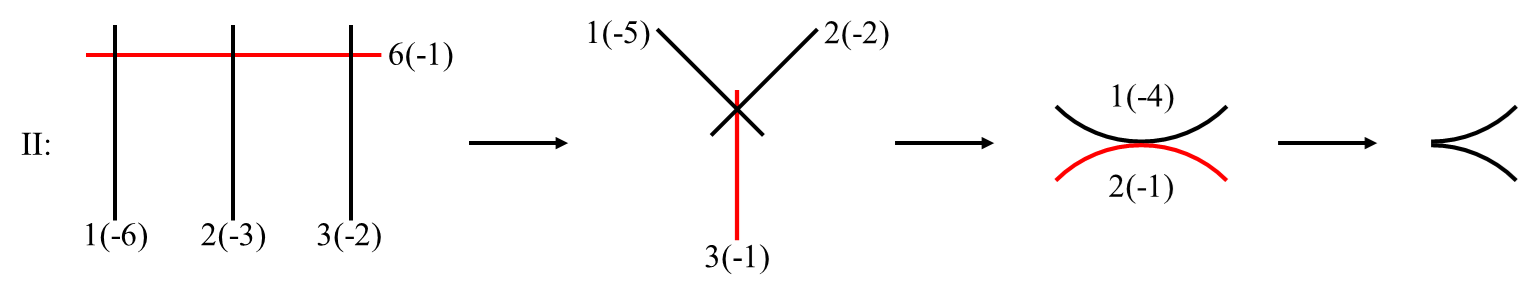}
    \includegraphics[height=2.4cm]{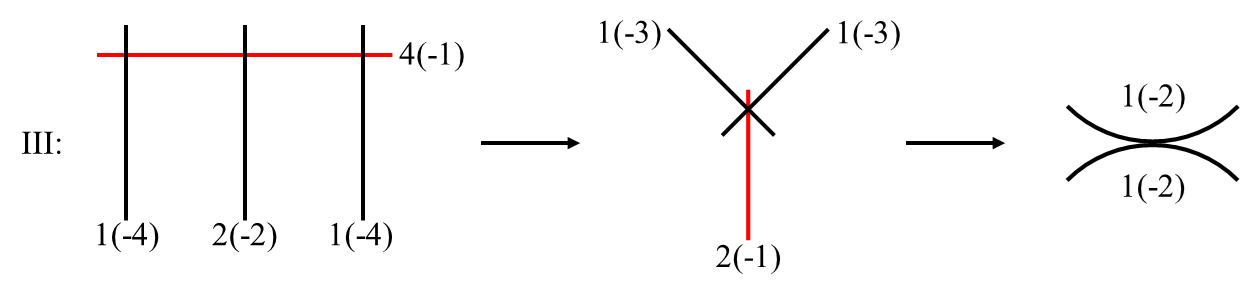}
    \includegraphics[height=2.4cm]{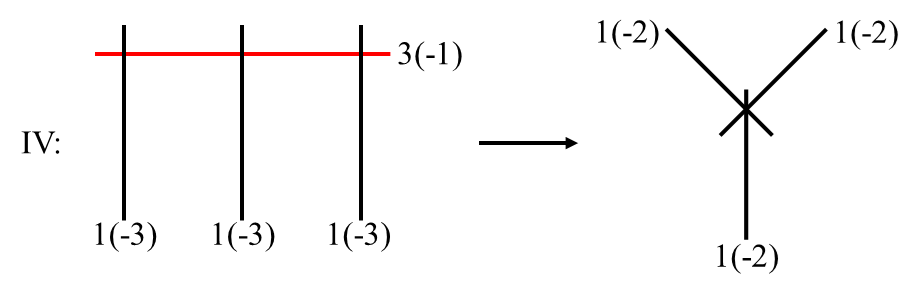}
    \caption{Blowing-down procedures. The notation $a(-b)$ indicates that the corresponding curve has multiplicity $a$ and self-intersection $-b$. The curves contracted in each procedure are colored red.}
    \label{fig:blowing-down-procedures}
    \centering
\end{figure}


\section{The non-vanishing conjecture for cotangent bundles} \label{sec:the-non-vanishing-conjecture}

In this section, we prove Theorem~\ref{thm:main-theorem}.
Let $f:S \rightarrow B$ be a relatively minimal isotrivial elliptic surface.
Write
 \[
  D = \sum_{b \in B} [f^{-1}(b) - f^{-1}(b)_{\mathrm{red}}] = \sum_{i=1}^s (\nu_i-1)F_i + D_0,
 \]
where $F_i$ is the reduction of a multiple fibre, $\nu_i$ is the multiplicity of $F_i$ and $D_0$ is the non-multiple, non-reduced part.
By Zariski's lemma \cite[Lemma~III.8.2]{BHPVdV2004}, the intersection matrix of $D_0$ is negative definite.
Thus, by \cite[Lemma~14.10]{Badescu2001}, $f^{\ast}\Omega_B^1(D)$ is pseudoeffective if and only if $f^{\ast}\Omega_B^1 \otimes \Ocal_S(\sum_{i=1}^s(\nu_i-1)F_i)$ is pseudoeffective.

Assume that $f^{\ast}\Omega_B^{1}(D)$ is not pseudoeffective, so that $B \simeq \PP^1$ and $\sum_{i=1}^s (1-\frac{1}{\nu_i}) < 2$.
Under the setting (\ref{eqn:isotrivial-elliptic-fibration}), let $Z_{\ast} \subset C$ be the set of points under which fibres are singular of type $\ast \in \{ I_0^{\ast}$, $II$, $II^{\ast}$, $III$, $III^{\ast}$, $IV$, $IV^{\ast} \}$.
Let $e_{\ast}$ be the ramification index of the quotient map $\pi:C \rightarrow B$ at a point $x \in Z_{\ast}$, namely
 \begin{equation} \label{eqn:ramification-index}
  e_{\ast} = |G_x| = \begin{cases}
      2, & \text{if } \ast = I_0^{\ast}, \\
      6, & \text{if } \ast = II \text{ or } II^{\ast}, \\
      4, & \text{if } \ast = III \text{ or } III^{\ast}, \\
      3, & \text{if } \ast = IV \text{ or } IV^{\ast}.
  \end{cases}
 \end{equation}

\begin{lem} [cf. {\cite[Lemma~6.4]{HP2020}}] \label{lem:number-of-ramification-points}
    Under the above setting, we have
     \begin{equation} \label{eqn:number-of-ramification-points}
      \sum (e_{\ast}-1)|Z_{\ast}| \geq 2g(C)-1.
     \end{equation}
\end{lem}

\begin{proof}
    By Riemann-Hurwitz formula, we have
     \[
      2g(C)-2 = d(2g(B)-2) + \sum_{x \in C} (e_x-1)
     \]
    where $d = \deg(\pi)$ and $e_x$ is the ramification index at $x$.
    Note that $e_x = |G_x|$.
    If the stabilizer group $G_x$ of $x \in C$ acts on $E$ by translation, then its order coincides with the multiplicity of the multiple fibre of $f$ under $x$.
    This property holds for every point in the orbit $G.x$.
    It follows that
     \[
      2g(C)-2 = d(2g(B)-2) + \sum (e_{\ast}-1)|Z_{\ast}| + \sum_{i=1}^s d \left( 1-\frac{1}{\nu_i} \right).
     \]
    Since $B \simeq \PP^1$ and $\sum_{i=1}^s (1-\frac{1}{\nu_i}) < 2$, this yields the desired inequality.
\end{proof}


\subsection{Logarithmic symmetric differentials} \label{subsec:log-symmetric-differential}

First, we will address the singular fibres of type $I_0^{\ast}$, $II$, $III$ and $IV$.
Consider $x \in Z_{III}$, and write $G_x = \langle g \rangle \simeq \ZZ/4\ZZ$.
Choose a small disc $x \in \Delta \subset C$ such that $(C \times E)/G$ is locally isomorphic to $(\Delta \times E)/G_x$.
Replacing the origin of $E$ if necessary, the fixed points of $g$ are $p_0 = (x,[0])$ and $p_1 = (x,[\frac{1}{2}+\frac{1}{2}i])$; the fixed point of $g^2$ is $p_2 = (x,[\frac{1}{2}])$.
The quotient singularities of $(\Delta \times E)/G_{x}$ at $q_{\nu} = q(p_{\nu})$ for $\nu = 0,1$ are $A_{4,1}$-singularities, so they are resolved by $(-4)$-curves $\Theta_{\nu}$; that of $q_2 = q(p_2)$ is an $A_1$-singularity, so it is resolved by a $(-2)$-curve $\Theta_2$.
Let $\Theta \subset S'|_{\Delta}$ be the proper transform of the central fibre of $(\Delta \times E)/G_{x}$ under $x$.
As shown by Figure~\ref{fig:blowing-down-procedures}, the $f'$-fibre under $x$ is $4\Theta + \Theta_0 + \Theta_1 + 2\Theta_2$ whose intersection numbers are
 \[
  (\Theta^2) = -1, \quad (\Theta_0^2) = (\Theta_1^2) = -4, \quad (\Theta_2^2) = -2, \quad
  (\Theta . \Theta_{\nu}) = 1.
 \]

To calculate the symmetric differentials on the surface $S'|_{\Delta}$, we first determine those admitting (at most) logarithmic poles along $\Theta_{\nu}$.
Let $E_0 = \Theta_0 + \Theta_1 + \Theta_2$. Define
 \[
  \Bcal_{\hbf} = S^{m}(\Omega_{S'}^1(\log E_0))(-\hbf \cdot E_0)
 \]
where $\hbf \in \ZZ_{\geq 0}^3$ and $\hbf \cdot E_0 = \sum_{\nu = 0}^2 h_{\nu}\Theta_{\nu}$.

\begin{lem} [cf. {\cite[Lemma~3.2]{BTVA2022}}] \label{lem:log-symmetric-differential-on-III}
    For $\hbf \in \ZZ_{\geq 0}^3$, we have
     \begin{equation} \label{eqn:symmetric-differentials-at-most-log-poles}
      H^0(S'|_{\Delta} \setminus E_0, \Bcal_{\hbf}) = H^0(S'|_{\Delta} \setminus E_0, S^m \Omega_{S'}^1) \simeq H^0(S'|_{\Delta}, \Bcal_{\hbf})
     \end{equation}
    if and only if $h_{\nu} < \frac{3m+1}{4}$ for $\nu = 0,1$ and $h_2 < \frac{m+1}{2}$.
\end{lem}

From the isomorphism
 \begin{equation} \label{eqn:isomorphism-of-log-differential}
 \begin{split}
  H^0(\Delta \times E, S^{m}\Omega_{C \times E}^1)^{G_x}
    &\simeq H^0(\Delta \times E \setminus \{ p_0,p_1,p_2 \}, S^{m}\Omega_{C \times E}^1)^{G_x} \\
    &\simeq H^0(S'|_{\Delta} \setminus E_0, S^m \Omega_{S'}^1),
 \end{split}
 \end{equation}
which is true as $\Delta \times E$ is smooth and $S^{m}\Omega_{C \times E}^1$ is reflexive, one can identify the space of logarithmic symmetric differentials on $S'$ as a subspace of symmetric differentials on $\Delta \times E$.

\begin{proof} [Proof of Lemma~\ref{lem:log-symmetric-differential-on-III}]
    Choose neighborhoods $\Delta_{\nu}$ of $p_{\nu}$ with respective local coordinates $(s,c_{\nu})$ such that $g:(s,c_{\nu}) \mapsto (\eta_4 s,\eta_4 c_{\nu})$ for $\nu=0,1$ and $g^2:(s,c_2) \mapsto (-s,-c_2)$.
    Then by \cite[II, p.583]{Kodaira1963} there exist coordinate charts $W_{\nu 1}$ and $W_{\nu 2}$ on $S'|_{\Delta}$ whose respective local coordinates are $(y_{\nu}, s_{\nu})$ and $(x_{\nu}, t_{\nu})$ such that
     \[
      \begin{cases}
          s^4 = y_{\nu}s_{\nu}^4 = x_{\nu}, \\
          c_{\nu}^4 = y_{\nu} = x_{\nu}t_{\nu}^4
      \end{cases}
     \]
    for $\nu = 0,1$ and
     \[
      \begin{cases}
          s^2 = y_{2}s_{2}^2 = x_2, \\ c_{2}^2 = y_{2} = x_{2}t_{2}^2.
      \end{cases}
     \]
    In those charts, $\Theta_{\nu}$ is defined by $y_{\nu} = x_{\nu} = 0$ and $\Theta$ is defined by $s_{\nu} = 0$.
    Moreover, $\Theta$ does not meet $W_{\nu 2}$.
    We visualize the setting in Figure~\ref{fig:figure_III_charts}.
    
    \begin{figure}[ht]
    \includegraphics[height=3.6cm]{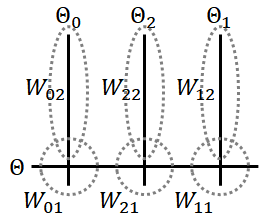}
    \caption{Coordinate charts on the $f'$-fibre of type $III$.}
    \label{fig:figure_III_charts}
    \centering
    \end{figure}
    
    For $\nu = 0,1$, following \cite[Section~3.2]{BTVA2022}, one can describe the valuation $\mathrm{ord}_{\Theta_{\nu}}$ on
    \[
        \bigoplus_{m \in \ZZ_{\geq 0}} H^0(\Delta_{\nu}, S^m \Omega_{C \times E}^1) \simeq \CC[s,c_{\nu}, ds, dc_{\nu}]
    \]
    as follows. It is expressed using the ring map
     \[
      \CC[s,c_{\nu},ds,dc_{\nu}] \rightarrow \CC (y_{\nu}^{1/4})[s_{\nu}, dy_{\nu}, ds_{\nu}]
     \]
    given by
     \[
      s \mapsto y_{\nu}^{1/4}s_{\nu}, \qquad c_{\nu} \mapsto y_{\nu}^{1/4},
     \]
     \[
      ds \mapsto 4^{-1}y_{\nu}^{-3/4}(4y_{\nu}ds_{\nu} + s_{\nu}dy_{\nu}), \qquad
      dc_{\nu} \mapsto 4^{-1}y_{\nu}^{-3/4}dy_{\nu}
     \]
    on $W_{\nu 1}$, and using the ring map
     \[
      \CC[s,c_{\nu},ds,dc_{\nu}] \rightarrow \CC (x_{\nu}^{1/4})[t_{\nu}, dx_{\nu}, dt_{\nu}]
     \]
    given by
     \[
      s \mapsto x_{\nu}^{1/4}, \qquad c_{\nu} \mapsto x_{\nu}^{1/4}t_{\nu},
     \]
     \[
      ds \mapsto 4^{-1}x_{\nu}^{-3/4}dx_{\nu}, \qquad
      dc_{\nu} \mapsto 4^{-1}x_{\nu}^{-3/4}(4x_{\nu}dt_{\nu} + t_{\nu}dx_{\nu})
     \]
    on $W_{\nu 2}$.
    Hence the valuation $\mathrm{ord}_{\Theta_{\nu}}$ satisfies
     \[
      \mathrm{ord}_{\Theta_{\nu}}(s) = \mathrm{ord}_{\Theta_{\nu}}(c_{\nu}) = \frac{1}{4}, \qquad
      \mathrm{ord}_{\Theta_{\nu}}(ds) = \mathrm{ord}_{\Theta_{\nu}}(dc_{\nu}) = -\frac{3}{4}.
     \]
    Since $H^0(W_{\nu 1}, S^{m}(\Omega_{S'}^1(\log E_0)))$ is the free $\CC[y_{\nu}, s_{\nu}]$-module generated by
     \[
      \left( \frac{dy_{\nu}}{y_{\nu}} \right)^m, \left( \frac{dy_{\nu}}{y_{\nu}} \right)^{m-1}ds_{\nu}, \dots, ds_{\nu}^{m}
     \]
    and $H^0(W_{\nu 2}, S^{m}(\Omega_{S'}^1(\log E_0)))$ is the free $\CC[x_{\nu}, t_{\nu}]$-module generated by
     \[
      \left( \frac{dx_{\nu}}{x_{\nu}} \right)^m, \left( \frac{dx_{\nu}}{x_{\nu}} \right)^{m-1}dt_{\nu}, \dots, dt_{\nu}^{m},
     \]
    one can see that $\frac{dy_{\nu}}{y_{\nu}}$ does not define a logarithmic symmetric differential on $W_{\nu 1} \cup W_{\nu 2}$, whilst $y_{\nu}(\frac{dy_{\nu}}{y_{\nu}})^4$ does.
    
    Similarly, for $\nu = 2$, the valuation $\mathrm{ord}_{\Theta_2}$ is expressed using
    \[
        \CC[s,c_2,ds,dc_2] \rightarrow \CC(y_2^{1/2})[s_2, dy_2, ds_2]
    \]
    given by
     \[
      s \mapsto y_{2}^{1/2}s_{2}, \qquad c_{2} \mapsto y_{2}^{1/2},
     \]
     \[
      ds \mapsto 2^{-1}y_{2}^{-1/2}(2y_{2}ds_{2} + s_{2}dy_{2}), \qquad
      dc_{2} \mapsto 2^{-1}y_{2}^{-1/2}dy_{2}
     \]
    on $W_{21}$, and using the ring map
     \[
      \CC[s,c_{2},ds,dc_{2}] \rightarrow \CC (x_{2}^{1/2})[t_{2}, dx_{2}, dt_{2}]
     \]
    given by
     \[
      s \mapsto x_{2}^{1/2}, \qquad c_{2} \mapsto x_{2}^{1/2}t_{2},
     \]
     \[
      ds \mapsto 2^{-1}x_{2}^{-1/2}dx_{2}, \qquad
      dc_{2} \mapsto 2^{-1}x_{2}^{-1/2}(2x_{2}dt_{2} + t_{2}dx_{2})
     \]
    on $W_{22}$.
    Thus one can see that $\frac{dy_2}{y_2}$ does not define a logarithmic symmetric differential on $W_{21} \cup W_{22}$, while $y_{2}(\frac{dy_2}{y_2})^2$ does.
    In this fashion, one concludes that
     \[
      H^0(S'|_{\Delta}, S^{m}(\Omega_{S'}^1(\log E_0)))
        \subset \bigoplus_{\nu} H^0(W_{\nu 0}, S^{m}(\Omega_{S'}^1(\log E_0)))
     \]
    consists of forms for which the coefficients of $(\frac{dy_{\nu}}{y_{\nu}})^{\ell} ds_{\nu}^{m-\ell}$ are divisible by $y_{\nu}^{\lceil \ell/4 \rceil}$ if $\nu = 0,1$; those of $(\frac{dy_{2}}{y_{2}})^{\ell} ds_{2}^{m-\ell}$ are divisible by $y_2^{\lceil \ell/2 \rceil}$.
    This proves Lemma~\ref{lem:log-symmetric-differential-on-III}.
\end{proof}

The above anaylsis yields similar results for the singular fibres of types $I_0^{\ast}$, $II$ and $IV$.
The result for $I_0^{\ast}$ is already known in \cite{BTVA2022}, but we include it for the sake of completeness.

\begin{lem} \label{lem:log-symmetric-differential}
    Let $x \in Z_{\ast}$, and choose a small disc $x \in \Delta \subset C$ such that $(C \times E)/G$ is locally isomorphic to $(\Delta \times E)/G_x$.
    Let $\Theta$ be the proper transform of the central fibre of $(\Delta \times E)/G_x$.
    Then we have
    \begin{enumerate}
        \item If $\ast = I_0^{\ast}$, the $f'$-fibre under $x$ is $2\Theta + \Theta_0 + \Theta_1 + \Theta_2 + \Theta_3$ with self-intersection $-2$ and $(\Theta.\Theta_{\nu}) = 1$ for all $\nu$.
        For a tuple $\hbf \in \mathbf{Z}_{\geq 0}^{4}$, we have
         \[
          H^0(S'|_{\Delta} \setminus E_0, S^m(\Omega_{S'}^1)) \simeq H^0(S'|_{\Delta}, S^m(\Omega_{S'}^1(\log E_0))(-\hbf \cdot E_0))
         \]
        if and only if $h_{\nu} < \frac{m+1}{2}$ for all $\nu$, where $E_0 = \sum_{\nu} \Theta_{\nu}$.
        \item If $\ast = II$, the $f'$-fibre under $x$ is $6\Theta + \Theta_0 + 2\Theta_1 + 3\Theta_2$ with self-intersection $(\Theta^2) = -1$, $(\Theta_{\nu}^2) = -6/(\nu+1)$ and $(\Theta.\Theta_{\nu}) = 1$ for all $\nu$.
        For a tuple $\hbf \in \mathbf{Z}_{\geq 0}^{3}$, we have
         \[
          H^0(S'|_{\Delta} \setminus E_0, S^m(\Omega_{S'}^1)) \simeq H^0(S'|_{\Delta}, S^m(\Omega_{S'}^1(\log E_0))(-\hbf \cdot E_0))
         \]
        if and only if $h_{\nu} < \frac{(6/(\nu+1)-1)m+1}{6/(\nu+1)}$ for all $\nu$, where $E_0 = \sum_{\nu} \Theta_{\nu}$.
        \item If $\ast = IV$, the $f'$-fibre under $x$ is $3\Theta + \Theta_0 + \Theta_1 + \Theta_2$ with self-intersection $(\Theta^2) = -1$, $(\Theta_{\nu}^2) = -3$ and $(\Theta.\Theta_{\nu}) = 1$ for all $\nu$.
        For a tuple $\hbf \in \mathbf{Z}_{\geq 0}^{3}$, we have
         \[
          H^0(S'|_{\Delta} \setminus E_0, S^m(\Omega_{S'}^1)) \simeq H^0(S'|_{\Delta}, S^m(\Omega_{S'}^1(\log E_0))(-\hbf \cdot E_0))
         \]
        if and only if $h_{\nu} < \frac{2m+1}{3}$ for all $\nu$, where $E_0 = \sum_{\nu} \Theta_{\nu}$.
    \end{enumerate}
\end{lem}


\subsection{Local obstructions} \label{subsec:local-obstructions}

Now let $A_C$ be an ample divisor on $C$ and $A_E$ be an ample divisor of degree one on $E$.
Set $A = A_C \boxtimes A_E$ on $C \times E$.
Choose a basis $s_{j,0}, s_{j,1}, \dots, s_{j,j-2}, s_{j,j}$ of $H^0(E, \Ocal_E(jA_E))$ such that each $s_{j,k}$ has vanishing order exactly $k$ at the origin $e \in E$.

As before, consider $x \in Z_{III}$ and choose a disc $x \in \Delta \subset C$ small enough.
Let $g$ be a generator of $G_x \simeq \ZZ/4\ZZ$, and $p_0 \in \Delta \times E$ a fixed point of $G_x$.
Then the singularity at $p_0$ is resolved by a $(-4)$-curve $\Theta_0$.
Choose a local coordinate $(s,c_0)$ at $p_0$ where $g$ acts by the multiplication by $\eta_4$.
We have a decomposition
 \[
  H^0(\Delta \times E, S^m\Omega_{C \times E}^1 \otimes \Ocal_{C \times E}(jA)) = \bigoplus_n V_{m,n}
 \]
where
 \[
  V_{m,n} = \langle s^{n-k}s_{j,k}ds^{\ell}dc_0^{m-\ell}:k = 0,1,\dots,j-2,j,\, 0 \leq k \leq n,\, 0 \leq \ell \leq m \rangle.
 \]
For $\omega \in H^0(\Delta \times E, S^m\Omega_{C \times E}^1 \otimes \Ocal_{C \times E}(jA))$, write $\omega = \sum_n \omega_n$ with $\omega_n \in V_{m,n}$.

Choose an ample divisor $A_{S'}$ on $S'$ and an ample Cartier divisor $\overline{A}$ on $(C \times E)/G$ such that there exist injective maps $\Ocal_{S'}(A_{S'}) \hookrightarrow \Ocal_{S'}(\lambda^{\ast}\overline{A})$ and $\Ocal_{C \times E}(q^{\ast}\overline{A}) \hookrightarrow \Ocal_{C \times E}(NA)$ for some $N \gg 0$.
Then there exists an injection
 \[
  \Phi:H^0(S',S^{m}\Omega_{S'}^1 \otimes \Ocal_{S'}(jA_{S'})) \hookrightarrow H^0(C \times E, S^m\Omega_{C \times E}^1 \otimes \Ocal_{C \times E}(NjA)),
 \]
as well as
 \[
  \Phi_{\Delta}:H^0(S'|_{\Delta},S^{m}\Omega_{S'}^1 \otimes \Ocal_{S'}(jA_{S'})) \hookrightarrow H^0(\Delta \times E, S^m\Omega_{C \times E}^1 \otimes \Ocal_{C \times E}(NjA)).
 \]
Remark that if $\omega \in \mathrm{im}(\Phi_{\Delta})$, then so is $\omega_n$ for each $n$ by \cite[Prop~3.3]{BTVA2022}.

\begin{lem}[cf. {\cite[Corollary~6.11]{HP2020}}] \label{lem:vanishing-order-of-symmetric-differential}
    For $\omega \in \mathrm{im}(\Phi_{\Delta})$, we have
     \[
      \omega \in H^0(\Delta \times E, I_{(x,e)}^n \otimes S^m\Omega_{C \times E}^1 \otimes \Ocal_{C \times E}(NjA))
     \]
    with $n \geq 3m-4Nj$, where $I_{(x,e)}$ is the ideal sheaf of $(x,e) \in \Delta \times E$.
\end{lem}

\begin{proof}
    Recall from the proof of Lemma~\ref{lem:log-symmetric-differential-on-III} that there exists a coordinate chart $W_{01}$ on $S'|_{\Delta}$ with local coordinate $(y_0,s_0)$ such that
    \[
        \begin{cases}
            s^4 = y_0s_0^4, \\
            c_0^4 = y_0
        \end{cases}
    \]
    and $\Theta_0$ is defined by $y_0 = 0$.
    In the chart $W_{01}$, we have
     \[
      s^{n-k}s_{j,k}ds^{\ell}dc_{0}^{m-\ell} = 4^{-m}y_0^{\frac{n-3m}{4}}s_0^{n-k}(s_0dy_0 + 4y_0ds_0)^{\ell}dy_0^{m-\ell} + O(y_0^{\frac{n-3m}{4}+1}).
     \]
    Thus by observing the leading terms with respect to $y_0$ and $ds_0$, one can deduce that for any $\omega \in V_{m,n}$, $\mathrm{ord}_{\Theta_0}(\omega) > \frac{n-3m}{4}$ if and only if $y_0ds_0 = c_0ds-sdc_0$ divides $\omega$.
    Hence for $\omega_n \in V_{m,n} \cap \mathrm{im}(\Phi_{\Delta})$, we have
     \[
      \omega_n = \begin{cases}
          \eta_n \times (sdc_0 - s_{1,1}ds)^{\frac{3m-n}{4}}, & \text{if } \frac{3}{5}m \leq n < 3m, \\
          0, & \text{if } n < \frac{3}{5}m
      \end{cases}
     \]
    for some $\eta_n \in H^0(\Delta \times E, S^{\frac{m+n}{4}}\Omega_{C \times E}^1 \otimes \Ocal_{C \times E} ((Nj - \frac{3m-n}{4}) A))$.
    Now since $H^0(E, \Ocal_E((Nj - \frac{3m-n}{4}) A_E)) = 0$ if $Nj < \frac{3m-n}{4}$, the form $\omega_n$ is nonzero only when $n \geq 3m-4Nj$.
\end{proof}


\subsection{The case of $II^{\ast}$, $III^{\ast}$ and $IV^{\ast}$} \label{subsec:local-obstructions}

Consider $x \in Z_{III^{\ast}}$ and write $G_x = \langle g \rangle \simeq \mathbf{Z}/4\mathbf{Z}$.
Again, choose a small disc $x \in \Delta \subset C$ as before.
Replacing the origin of $E$ if necessary, the fixed points of $g$ are $p_0 = (x,[0])$ and $p_1 = (x,[\frac{1}{2} + \frac{1}{2}i])$; the fixed points of $g^2$ is $p_2 = (x, [\frac{1}{2}])$.
One can choose local coordinates $(s,c_{\nu})$ at $p_{\nu}$ such that $g:(s,c_{\nu}) \mapsto (\eta_4 s,\eta_4^{-1}c_{\nu})$ for $\nu = 0,1$ and $g^2:(s,c_2) \mapsto (-s,-c_2)$.
Hence the quotient $(\Delta \times E)/G_x$ has an $A_3$-singularity at $q_{\nu} = \pi(p_{\nu})$ if $\nu = 0,1$ and an $A_1$-singularity at $q_2 = \pi(p_2)$.
Analogously to Lemma~\ref{lem:log-symmetric-differential-on-III}, we have:

\begin{thm}[{\cite[Theorem~3.(a)]{AOW2023}}]
    Suppose that a normal surface $S$ has an $A_n$-singularity at $0 \in S$.
    Let $\lambda:(\widetilde{S},E_0) \rightarrow (S,0)$ be the minimal resolution.
    Then for a tuple $\hbf = (h_1, \dots, h_n) \in \ZZ_{\geq 0}^n$, we have
     \[
      H^0(\widetilde{S} \setminus E_0, S^m \Omega_{\widetilde{S}}^1) \simeq H^0(\widetilde{S}, S^m \Omega_{\widetilde{S}}^1(\log E_0) \otimes \Ocal_{\widetilde{S}}(-\hbf \cdot E_0))
     \]
    if and only if
     \begin{equation} \label{eqn:order-of-symmetric-differentials-on-an-singularity}
      h_i \leq \sum_{j=0}^{\min \{ i-1, n-i \}} \left\lceil \frac{m-2j}{n+1} \right\rceil
     \end{equation}
    for each $i$, where $E_0 = \Theta_1 + \cdots + \Theta_n$ is the exceptional locus with $(\Theta_i.\Theta_{i+1}) = 1$ and
     \[
      \hbf \cdot E_0 = \sum_{i=1}^n h_i \Theta_i.
     \]
\end{thm}

It only remains to examine the local obstruction for symmetric differentials.
Each quotient singularity of $(\Delta \times E)/G_x$ at $q_{\nu}$ is resolved by a chain of rational curves $\Theta_{\nu 1}$, $\Theta_{\nu 2}$ and $\Theta_{\nu 3}$ if $\nu = 0,1$ and by a rational curve $\Theta_{2}$ if $\nu = 2$.
Let $\Theta$ be the proper transform of the central fibre of $(\Delta \times E)/G_x$.
Then the $f'$-fibre under $x$ is given by
 \[
  4\Theta + 3\Theta_{01} + 2\Theta_{02} + \Theta_{03} + 3\Theta_{11} + 2\Theta_{12} + \Theta_{13} + 2\Theta_2
 \]
with self-intersection numbers $-2$ and
 \begin{align*}
  (\Theta.\Theta_{01}) &= (\Theta_{01}.\Theta_{02}) = (\Theta_{02}.\Theta_{03}) = (\Theta.\Theta_{11}) \\
    &= (\Theta_{11}.\Theta_{12}) = (\Theta_{12}.\Theta_{13}) = (\Theta.\Theta_{2}) = 1.
 \end{align*}
From \cite[II, p.584]{Kodaira1963}, there exist coordinate charts $W_{\nu 1}$, $W_{\nu 2}$, $W_{\nu 3}$, $W_{\nu 4}$, $W_{21}$ and $W_{22}$ with respective local coordinates $(y_{\nu}, s_{\nu 1})$, $(t_{\nu 1},s_{\nu 2})$, $(t_{\nu 2},s_{\nu 3})$, $(x_{\nu},t_{\nu 3})$, $(y_2, s_2)$ and $(x_2, t_2)$ such that
 \[
  \begin{cases}
      s^4 = y_{\nu}^3s_{\nu 1}^4 = t_{\nu 1}^2s_{\nu 2}^3 = t_{\nu 2}s_{\nu 3}^2 = x_{\nu 1}, \\
      c_{\nu}^4 = y_{\nu} = t_{\nu 1}^2s_{\nu 2} = t_{\nu 2}^3s_{\nu 3}^2 = x_{\nu 1}^3t_{\nu 3}^4, \\
      s_2^2 = y_2s_2^2 = x_2, \\
      c_2^2 = y_2 = x_2t_2^2.
  \end{cases}
 \]

\begin{figure}[ht]
    \includegraphics[height=4.3cm]{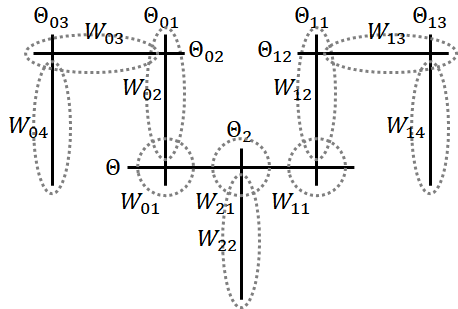}
    \caption{Coordinate charts on the $f'$-fibre of type $III^{\ast}$.}
    \label{fig:figure_III_star_charts}
    \centering
\end{figure}

\noindent In those charts, $\Theta_{\nu 1}$, $\Theta_{\nu 2}$ and $\Theta_{\nu 3}$ are defined by $y_{\nu} = s_{\nu 2} = 0$, $t_{\nu 1} = s_{\nu 3} = 0$ and $t_{\nu 2} = x_{\nu} = 0$ respectively; $\Theta_2$ is defined by $y_2 = x_2 = 0$; $\Theta$ is defined by $s_{\nu 1} = s_2 = 0$.
Also, $\Theta$ does not meet $W_{\nu 2}$, $W_{\nu 3}$, $W_{\nu 4}$ and $W_{22}$.
As before, we visualize the setting in Figure~\ref{fig:figure_III_star_charts}.

On the chart $W_{01}$, we have
 \begin{align*}
  \omega &:= s^{n-k}s_{j,k} ds^{\ell}dc_0^{m-\ell} \\
    &= 4^{-m}y_{01}^{\frac{k+3(n-k)}{4} - \frac{\ell + 3(m-\ell)}{4}}s_{01}^{n-k}(3s_{01}dy_{01} + 4y_{01}ds_{01})^{\ell}dy_{01}^{m-\ell} + O(y_{01}^{\frac{n-3m}{4}+1})
 \end{align*}
and its valuation with respect to $\Theta_{01}$ is at least $\frac{n-3m}{4}$.
If the valuation exceeds $\frac{n-3m}{4}$, then either $\ell > 0$ or $n-k > 0$.
One can see that $\omega|_{W_{01}}$ is divisible by $c_0ds = \frac{3}{4}s_{01}dy_{0} + y_{0}ds_{01}$ if $\ell > 0$, and by $sdc_0 = \frac{1}{4}s_{01}dy_{0}$ if $\ell = 0$ and $n-k > 0$.
The similar calculations for all the other curves $\Theta_{0i}$ yield the following:

\begin{lem} \label{lem:vanishing-order-of-symmetric-differential-for-star}
    For $\omega \in \mathrm{im}(\Phi_{\Delta})$, we have
     \[
      \omega \in H^0(\Delta \times E, I_{(x,e)}^n \otimes S^m\Omega_{C \times E}^1 \otimes \Ocal_{C \times E}(NjA))
     \]
    with $n \geq 3m-4Nj$.
\end{lem}

In summary, one can find a bound for the vanishing order of a symmetric differential on $S'$ in terms of the types of singular fibres.

\begin{cor} \label{cor:vanishing-order-of-symmetric-differential}
    If $\omega \in \mathrm{im}(\Phi)$, then for each $x \in Z_{\ast}$, there exists a point $e \in E$ such that
     \[
      \omega \in H^0(C \times E, I_{(x,e)}^n \otimes S^m\Omega_{C \times E}^1 \otimes \Ocal_{C \times E}(NjA))
     \]
    with $n \geq (e_{\ast}-1)m - e_{\ast}Nj$.
\end{cor}

\begin{proof}[Proof of Theorem~\ref{thm:main-theorem}]
    Assume on the contrary that $f^{\ast}\Omega_B^1(D)$ is not pseudoeffective.
    Then both $B \simeq \PP^1$ and $\sum_{i=1}^s (1-\frac{1}{\nu_i}) < 2$ holds.
    Thus $f'^{\ast}\Omega_B^1(D')$ is not pseudoeffective where
     \[
      D' = \sum_{b \in B} [f'^{-1}(b) - f'^{-1}(b)_{\mathrm{red}}]
     \]
    because multiple fibres of $f'$ do not contain $(-1)$-curves.
    
    Let $g = g(C)$ be the genus of $C$.
    Choose $\varepsilon \in \QQ_{> 0}$ and $N' \in \NN$ such that $\frac{2g-2}{2g-1} + \varepsilon < 1$ and $N'(\frac{2g-2}{2g-1} + \varepsilon) \in \NN$.
    Then for $m \geq \frac{2}{1-\frac{2g-2}{2g-1}-\varepsilon}Nj$, a symmetric differential $\omega \in \mathrm{im}(\Phi)$ induces
     \[
      \omega^{N'} \in H^0 \left( C \times E, \left( \bigotimes_{x \in Z_{\ast}} I_{(x,e)}^{\left( \frac{2g-2}{2g-1}+\varepsilon \right)N'm(e_{\ast}-1)} \right) \otimes S^{N'm}\Omega_{C \times E}^1 \otimes \Ocal_{C \times E}(N'NjA) \right)
     \]
    where for each $x \in C$, a point $e \in E$ is chosen such that $G_x$ fixes $e$.
    From Lemma~\ref{lem:number-of-ramification-points}, we infer that
     \begin{align*}
      \deg_C \left( \bigotimes_{x \in Z_{\ast}} I_{x}^{\left( \frac{2g-2}{2g-1}+\varepsilon \right)N'm(e_{\ast}-1)} \right)
        &= N'm \left( \frac{2g-2}{2g-1}+\varepsilon \right) \sum (e_{\ast}-1) |Z_{\ast}| \\
        &> N'm(2g-2).
     \end{align*}
    Hence the claim of the proof of \cite[Theorem~6.7]{HP2020} applies to prove that $\Phi = 0$, so $\Omega_{S'}^1$ is not pseudoeffective.
    Now $\Omega_S^1$ being pseudoeffective is equivalent to $\Omega_{S'}^1$ being pseudoeffective by \cite[Proposition~4.1]{HP2020}, the proof is complete.
\end{proof}

\bibliographystyle{amsplain}

\end{document}